\documentclass[10pt]{amsart}
\usepackage[a4paper,margin=2.5cm]{geometry}
\usepackage{hyperref,mathtools,amsmath,amsthm,amssymb,amscd,amsfonts,color,tikz-cd,pgfplots,diagbox,dsfont}
\usetikzlibrary{arrows}
\usepackage{tikz}
\usetikzlibrary{positioning}

\theoremstyle{plain}
\newtheorem{theorem}{Theorem}[section]
\newtheorem{proposition}[theorem]{Proposition}
\newtheorem{lemma}[theorem]{Lemma}
\newtheorem{corollary}[theorem]{Corollary}
\numberwithin{equation}{section}
\theoremstyle{definition}

\newtheorem{definition}[theorem]{Definition}
\newtheorem{remark}[theorem]{Remark}
\newtheorem{example}[theorem]{Example}
\newtheorem{question}[theorem]{Question}

\newtheorem{problem}[theorem]{Problem}

\newcommand{\CP}{\mathbb{C}P}
\newcommand{\RP}{\mathbb{R}P}

\newcommand{\cC}{\mathcal{C}}

\newcommand{\cI}{\mathcal{I}}

\newcommand{\cG}{\mathcal{G}}

\newcommand{\edge}{\mathbf{e}}
\newcommand{\EC}{\mathcal{E}\mathcal{C}}
\newcommand{\comp}{\overline{\EC}}
\newcommand{\B}{\mathcal{B}}

\newcommand{\Q}{\mathbb{Q}}
\newcommand{\F}{\mathbb{F}}
\newcommand{\R}{\mathbb{R}}
\newcommand{\Z}{\mathbb{Z}}

\newcommand{\ie}{\textit{i}.\textit{e}. }

\DeclareMathOperator{\conv}{conv}

\DeclareMathOperator{\cone}{cone}

\begin{document}
\title[A decomposition of graph $a$-numbers]{A decomposition of graph $a$-numbers}

\author[S. Choi]{Suyoung Choi}
\address{Department of mathematics, Ajou University, 206, World cup-ro, Yeongtong-gu, Suwon 16499,  Republic of Korea}
\email{schoi@ajou.ac.kr}

\author[Y. Yoon]{Younghan Yoon}
\address{Department of mathematics, Ajou University, 206, World cup-ro, Yeongtong-gu, Suwon 16499,  Republic of Korea}
\email{younghan300@ajou.ac.kr}

\date{\today}
\subjclass[2020]{14M25, 55N10, 14P25, 57S12, 05C30}

\keywords{graph $a$-numbers, cohomology, Betti numbers, real toric varieties, graphs, unimodality, log-concavity, blow-ups, Lefschetz operators}

\thanks{This work was supported by the National Research Foundation of Korea Grant funded by
the Korean Government (RS-2025-00521982).}

\begin{abstract}
We study the $a$-sequence $(a_0(G), a_1(G), \cdots)$ of a finite simple graph $G$, defined recursively through a combinatorial rule and known to coincide with the sequence of rational Betti numbers of the real toric variety associated with $G$.

In this paper, we establish a combinatorial and topological decomposition formula for the $a$-sequence.
As an application, we show that the $a$-sequence is monotone under graph inclusion; that is, $a_i(G) \geq a_i(H)$ for all $i \geq 0$ whenever $H$ is a subgraph of $G$, and obtain the lower and upper bounds of $a_i$-numbers.
We also prove that the $a$-sequence is unimodal in $i$ for a broad class of graphs $G$, including those with a Hamiltonian circuit or a universal vertex. 
These results provide a new class of topological spaces whose Betti number sequences are unimodal but not necessarily log concave, contributing to the study of real loci in algebraic geometry.    
\end{abstract}
\maketitle

\section{Introduction}\label{sec:intro}
The \emph{$a$-number} is a graph invariant introduced by Choi and Park \cite{Choi-Park2015} to compute the rational Betti numbers of a canonical real toric variety $X^\R_G$ associated with a finite simple graph $G$. 
While its motivation is geometric, the $a$-number is defined purely in combinatorial terms, making it both topologically meaningful and computationally accessible. 
This fusion of geometry and combinatorics places the $a$-number at a crossroads of several active fields of research, ranging from real algebraic geometry to topological combinatorics.

Let $G$ be a finite simple graph with vertex set $V(G)$ and edge set $E(G)$.
The \emph{signed $a$-number}~$sa(G)$ of~$G$ is defined recursively as follows:
\begin{itemize}
    \item If $G=\emptyset$ is the null graph, then $sa(\emptyset) = 1$.
    \item If $G_1,\ldots, G_\ell$ are connected components of $G$, then 
    $$
        sa(G)=\prod_{i=1}^\ell sa(G_i).
    $$  
    \item If $G$ is connected, then
\begin{equation} \label{rec}
    sa(G) = \left\{
              \begin{array}{ll} 
                -\sum_{I\subsetneq V(G)}sa(G\vert_I), & \hbox{if $G$ has even order;} \\
                0, & \hbox{otherwise,}
              \end{array}
            \right.
\end{equation}
where $G\vert_I$ is the induced subgraph of $G$ by $I \subseteq V(G)$.
\end{itemize}
The \emph{$a$-number}~$a(G)$ of~$G$ is then defined as the absolute value of $sa(G)$, \ie, $a(G) := | sa(G) |$.

As originally intended, the $a$-number provides a purely combinatorial framework for computing homological invariants of $X^\R_G$, including both its rational Betti numbers and Euler characteristic, which are among the most fundamental topological quantities associated with a space.
By \cite{Choi-Park2015}, the $i$th rational Betti number of $X^\R_G$ is given by the \emph{$a_i$-number} of $G$, defined as
$$
    a_i(G):=\sum_{\substack{I\subseteq V(G)\\ |I| =2i}}a(G\vert_I),
$$
and the Euler characteristic of $X^\R_G$ is given by
$$
    b(G):=\sum_{I\subseteq V(G)}sa(G\vert_I).
$$
The sequence $(a_0(G), a_1(G), \ldots )$ of $G$ is called the \emph{$a$-sequence} of $G$.

When $G$ is a connected graph with an even number of vertices, the $a$-number and $a_i$-numbers can also be described via the M\"{o}bius function of a poset.
Consider the collection of vertex subsets whose induced subgraphs have no odd connected components. This collection, ordered by inclusion, forms a graded poset.
The signed $a$-number of $G$ is given by the M\"{o}bius function evaluated at the interval from the empty set to the full vertex set.
Moreover, the $i$th $a$-number corresponds to the absolute value of the $i$th Whitney number of the first kind of this poset.

The $a$-sequences are related to several well-known and combinatorially interesting sequences, depending on the structure of the underlying graph $G$.
For instance, when $G$ is a path graph, the $a$-number yields the Catalan numbers, and the $a_i$-numbers produce the Catalan triangle; see \cite{Choi-Park2015} and \cite{Stanley2015book-Catalan}.
When $G$ is a cycle graph, the resulting $a$-sequence forms the so-called half Pascal triangle.
The $a$-sequences can also be computed explicitly for complete graphs, star graphs, and more generally, complete bipartite graphs; see \cite{Choi-Park2015} and \cite{Seo-Shin2015} for detailed computations.
These connections illustrate the rich combinatorial structure encoded in the $a$-sequence and underscore its relevance in algebraic and enumerative combinatorics.

In this paper, we establish a structural decomposition of the $a_i$-number of a graph $G$, which reveals unexpected regularities, namely the monotonicity of the $a$-number and the unimodality of the $a_i$-numbers, that are not apparent from the recursive definition, and opens the way to attack two fundamental problems concerning the nature of these invariants.

To present a decomposition formula, we need to introduce some notations.
For each pair set~$\edge$ in~$V(G)$, $G + \edge$ is the graph obtained from~$G$ by adding~$\edge$ as an edge. 
If $\edge$ is already an edge of~$G$, then it does not change the graph, \ie, $G+\edge = G$.
We also consider the collection of vertex subsets that induced subgraphs with only even-order connected components, defined as
$$
    \EC(G) := \{I \subseteq V(G) \colon G\vert_I \text{ has no connected component of odd order}\}.
$$
Finally, we recall the notion of the \emph{reconnected complement}~$G^\ast_{I}$  of~$I$ in~$G$, introduced in~\cite{Carr-Devadoss2006}.
For a subset~$I \subseteq V(G)$, define the graph~$G^\ast_{I}$ on~$V(G) \setminus I$ as follows; for each unordered pairs~$\{a,b\} \subset V(G) \setminus I$, the pair $\{a,b\}$ is an edge of $G^\ast_{I}$ if and only if there exists a path from~$a$ to~$b$ in~$G_{I \cup \{a,b\}}$.
Further developments and applications of the reconnected complement can be found in~\cite{Buchstaber-Panov2015} and~\cite{Vladimir-Adam-Denis_2024}.

\begin{theorem}
Let $G$ be a finite simple graphs with~$n$ vertices, and let $\edge$ be a pair of vertices of~$G$.
For each $i \geq 0$, 
    $$
    a_i(G+\edge) = a_i(G) + \sum_{J \in \EC(G+\edge) \setminus \EC(G)}\left\vert  b(G\vert_J) \right\vert \cdot a_{i-\frac{|J|}{2}}((G +\edge)^\ast_{J}).
    $$
\end{theorem}

This decomposition enables us to explore two particularly intriguing problems related to the $a$-number and $a_i$-numbers.

The first problem concerns the \emph{monotonicity} of the $a$-sequences: does the $a_i$-number increase under inclusion of graphs for all $i \geq 0$?
Since the definition of the $a_i$-number involves an alternating sum over subgraphs, its monotonicity is far from obvious. 
Nevertheless, the decomposition reveals that all coefficients are non-negative, leading to the following corollary.

\begin{corollary}\label{coro:monoto}
  For all $i \geq 0$, the $i$th graph $a$-number~$a_i$ is monotonically increasing under graph inclusion; that is, if $H$ is a subgraph of~$G$, then
    $$
    a_i(H) \leq a_i(G).
    $$
    In particular, if $H$ is a spanning subgraph of $G$, then $a(H) \leq a(G)$.
\end{corollary}

As an immediate application of the monotonicity theorem, we obtain the first inequality in the following corollary, by comparing a graph with a spanning tree and the complete graph on the same vertex set.
The additional part is more subtle and the proof follows again the decomposition theorem, and we refer to Theorem~\ref{thm:bounds_a_i_numbers}.
\begin{corollary}\label{coro:upperlower}
  Let~$G$ be a connected graph on~$n$ vertices.
  Let~$K_n$ be the complete graph on~$n$ vertices, and let~$T$ be a spanning tree of~$G$.
  Then, for each~$i$,
  ~$$
    a_i(T) \le a_i(G) \le a_i(K_n) = \binom{n}{2i} A_{2i},
  $$
  where $A_{k}$ is the $k$th Euler zigzag number.
  Moreover, 
  $$
    \frac{n-2i+1}{n-i+1} \binom{n}{i}=a_i(P_n) \leq a_i (T) \leq a_i (S_n) = \binom{n-1}{2i-1} A_{2i-1}.
  $$
\end{corollary}

The second problem deals with the \emph{unimodality} of the $a_i$-numbers as a sequence in $i$. 
One of the most well-known conjectures related to unimodality is the $g$-conjecture, originally posed by McMullen~\cite{McMullen1971}, which characterizes the $h$-vectors of simplicial polytopes. 
It was resolved by Stanley~\cite{Stanley1980} in the necessity direction, by constructing a projective toric variety whose cohomology ring reflects the $h$-vector, and applying the Hard Lefschetz theorem.
This approach laid the foundation for the field of \emph{toric topology}, which seeks to understand topological invariants via associated combinatorial structures.
This unimodality has also been established for more general complex smooth compact toric varieties~\cite{Adiprasito2018,Karu-Xiao2023}.

However, this does not generally hold for the real loci of toric varieties. 
In our setting, the $a$-vector of a graph coincides with the rational Betti numbers of a corresponding real toric variety.
From the toric topological perspective, it is natural to ask whether these sequences exhibit unimodality, or more broadly, to seek necessary and sufficient conditions for a sequence of positive integers to arise as the $a$-vector of a finite simple graph, as formulated in Problem~\ref{prob:g-conjecture}.
Interestingly, for many graphs $G$, the real toric variety $X^\R_G$ appears to exhibit unimodal Betti numbers. 
This phenomenon was posed as an open question in the authors' previous work~\cite{Choi-Yoon2026}.
Our decomposition theorem allows us to establish unimodality for several concrete subclasses.
We record one such consequence below. 
See Corollaries~\ref{cor:Ham} and~\ref{cor:universal_vertex} for the details of the proof.

\begin{corollary}
    The $a$-sequence of a finite simple graph $G$ is unimodal if $G$ has a Hamiltonian circuit or a universal vertex.
\end{corollary}

It is also worth noting that these $a$-sequences need not be log-concave as shown in Remark~\ref{rem:non_log_concave_exam}. 
While topological spaces with log-concave Betti number sequences have received considerable attention (see, e.g., \cite{Huh-Katz2012,Adiprasito-Huh-Katz2018,Gui-Xiong2024}), much less is known about those that exhibit unimodality without log-concavity.
Although there are~\cite{Cairns2009,Migliore-Zanello2018} several known examples of Betti number sequences of chain complexes that are unimodal but not log-concave, such phenomena have been little studied for Betti number sequences associated with topological spaces.
Our results contribute to this direction by exhibiting a broad class of such topological spaces.

The paper is organized as follows.
In Section~\ref{sec:pre}, we introduce preliminary concepts.
In particular, we introduce the graph associahedra along with their associated complex and real toric varieties, as well as the notion of reconnected complements and certain graph classes that are closed under reconnected complements.
Section~\ref{sec:graph} presents properties of graph $a$-numbers and the motivation behind their decomposition.
In Section~\ref{sec:decom}, we prove a decomposition theorem for the $a$-number and, as an application, demonstrate its monotonicity.
In Section~\ref{sec:bounds}, we prove Corollary~\ref{coro:upperlower}, and Section~\ref{sec:unimodal} investigates the unimodality of the $a$-sequences.
In Subsection~\ref{subsec:blow-ups}, we study blow-ups of toric varieties, focusing on their real loci.
We raise several natural questions and give partial answers concerning how these blow-ups influence the topology of the corresponding real toric varieties.
In Subsection~\ref{subsec:Lefschetz}, we discuss why, in general, the cohomology of real toric varieties does not admit an operator analogous to the Lefschetz operator, even for those arising from graphs.

\section{Reconnected complements}\label{sec:pre}

Let $S$ be a finite set.
A finite collection $\B$ of nonempty subsets of~$S$ is called a \emph{building set} on~$S$ if it satisfies the following conditions:
\begin{enumerate}
    \item $\{i\} \in \B$ for every $i \in S$, and
    \item if $I, J \in \B$ and $I \cap J \neq \emptyset$, then $I \cup J \in \B$.
\end{enumerate}
An inclusion-maximal element of~$\B$ is called a \emph{connected component}, and~$\B$ is said to be \emph{connected} if~$S \in \B$.  
For a building set~$\B$ on~$[n+1] = \{1, \ldots, n+1\}$, the \emph{nestohedron}~$P_{\B}$ associated with~$\B$ is defined as the Minkowski sum of simplices
$$
    P_\B = \sum_{I \in \B} \conv(I),
$$
where~$\conv(I)$ denotes the convex hull of the $i$th standard basis vectors in~$\R^{n+1}$ for all~$i \in I$.

The nestohedron $P_\B$ associated with a building set~$\B$ naturally lies in~$\R^{n+1}$, but is contained in the hyperplane $\sum_{i=1}^{n+1} x_i = |\B|$. 
Modulo the line spanned by the vector~$(1,1,\ldots,1)$, this hyperplane can be identified with~$\R^n$ equipped with the standard lattice~$\Z^n$.
It is well known~\cite{Postnikov2009} that the image of any nestohedron under this projection is a \emph{Delzant polytope} in~$\R^n$, meaning that at each vertex~$v$, the outward normal vectors to the facets containing~$v$ can be chosen to form an integral basis of~$\Z^n$.
By the fundamental theorem of toric geometry, $P_\B$ induces a smooth projective toric variety~$X_{P_\B}$, called the \emph{toric variety associated with~$\B$}.
The fixed point set of~$X_{P_\B}$ under the canonical involution induced by a complex conjugation is denoted by~$X^\R_{P_\B}$, and is referred to as the \emph{real toric variety associated with~$\B$}.

A building set~$\B$ on a finite set~$S$ is called \emph{graphical} if there exists a finite simple graph~$G$ on~$S$ such that~$I \in \B$ if and only if the induced subgraph~$G|_I$ on~$I$ is connected.  
In this case, the building set~$\B$ is denoted by~$\B(G)$.
Note that a building set~$\B(G)$ is connected if and only if~$G$ is connected.

Given a finite simple graph~$G$, the associated nestohedron~$P_{P_\B(G)}$ is simply denoted by~$P_G$ and is called a \emph{graph associahedron}, which is a subclass of nestohedra and has been extensively studied in the literature~\cite{Carr-Devadoss2006, Zel2006, Postnikov2008, Devadoss2009, Choi-Park2015, Rodrigo-Jensen-Dhruv2016, Choi-Hwang2023}.
The associated complex and real toric varieties, denoted by~$X_G$ and~$X^\R_G$ respectively, are simply called the complex and real toric varieties associated with~$G$.

\begin{definition}
    For each $I \subseteq V(G)$, define the graph~$G^\ast_I$ whose vertex set is~$V(G) \setminus I$ and whose edge set consists of unordered pairs~$\{a,b\} \subset V(G) \setminus I$ such that there exists a path from~$a$ to~$b$ in~$G|_{I \cup \{a,b\}}$.  
    We call~$G^\ast_I$ the \emph{reconnected complement} of~$I$ in~$G$.
\end{definition}

Reconnected complements were studied in~\cite{Carr-Devadoss2006,Buchstaber-Panov2015,Vladimir-Adam-Denis_2024}.
\begin{proposition}\cite{Carr-Devadoss2006}\label{prop:carr}
    There is a one-to-one correspondence between the facets of~$P_G$ and
    $$
    \cI = \{I \subsetneq V(G) \colon I \neq \emptyset, G\vert_I \text{ is connected}\}.
    $$
    Moreover, the facet corresponding to~$I \in \cI$ is combinatorially equivalent to the product~$P_{G\vert_I} \times P_{G_I^\ast}$.
\end{proposition}

It is straightforward to verify that~$G^\ast_{\emptyset} = G$ and~$G^\ast_{V(G)}$ is the null graph.
Moreover, note that the set of edges~$E(G^\ast_{I})$ of $G^\ast_{I}$ can be classified into exactly two types.
One type consists of edges originally from~$G$, while the other consists of edges that are not in~$G$ but are added to~$G^\ast_I$ since they appear as paths in~$G_{I \cup \{\edge\}}$.

\begin{example}
    Let~$G$ be the graph obtained from a $6$-cycle graph with vertex set $\{1, 2, 3, 4, 5, 6\}$ in cyclic order by adding an edge~$\{1,4\}$.
    Let us consider the reconnected complement of~$I = \{1,2\}$ in~$G$.  
The edges~$\{3,4\}$, $\{4,5\}$, and $\{5,6\}$ remain as edges in~$G^\ast_{\{1,2\}}$, since they are already present in~$G$ and their endpoints are unaffected by the removal.  
In addition, paths exist from~$3$ to~$6$ in the induced subgraph~$G^\ast_{\{1,2,3,6\}}$ and from~$4$ to~$6$ in~$G^\ast_{\{1,2, 4,6\}}$.  
In contrast, there is no path from~$3$ to~$5$ in~$G|_{\{1,2, 3,5\}}$.  
Hence,  
$$
E(G^\ast_{\{1,2\}}) = \big\{\{3,4\}, \{4,5\}, \{5,6\}, \{3,6\}, \{4,6\}\big\}.
$$  
See Figure~\ref{figure:rec} for additional examples of reconnected complements, including the one above.
\begin{figure}
\centering
\begin{tikzpicture}[scale=0.7]
  \node (1) at (0, 3) {$1$};
  \node (2) at (-2, 2) {$2$};
  \node (3) at (-2, 0) {$3$};
  \node (4) at (0, -1) {$4$};
  \node (5) at (2, 0) {$5$};
  \node (6) at (2, 2) {$6$};
  
  \draw (1) -- (2);
  \draw (2) -- (3);
  \draw (3) -- (4);
  \draw (4) -- (5);
  \draw (5) -- (6);
  \draw (6) -- (1);
  \draw (4) -- (1);

  \node (3a) at (4, 0) {$3$};
  \node (4a) at (6, -1) {$4$};
  \node (5a) at (8, 0) {$5$};
  \node (6a) at (8, 2) {$6$};
  
  \draw (3a) -- (4a);
  \draw (4a) -- (5a);
  \draw (5a) -- (6a);
  \draw (6a) -- (3a);
  \draw (6a) -- (4a);

  \node (3b) at (10, 0) {$3$};
  \node (5b) at (14, 0) {$5$};
  \node (6b) at (14, 2) {$6$};
  
  \draw (3b) -- (5b);
  \draw (3b) -- (6b);
  \draw (5b) -- (6b);
  
  \node (3c) at (16, 0) {$3$};
  \node (6c) at (20, 2) {$6$};
  \draw (3c) -- (6c);
  
  \node at (0, -2) {$G$};
  \node at (6, -2) {$G^\ast_{\{1,2\}}$};
  \node at (12, -2) {$G^\ast_{ \{1,2,4\}}$};
  \node at (18, -2) {$G^\ast_{ \{1,2,4,5\}}$};
\end{tikzpicture}
  \caption{Reconnected complements}\label{figure:rec}
\end{figure}
\end{example}

In the rest of this section, we introduce new properties of reconnected complements, which plays a key role in proving the main result of this paper.
Basic observations are as follows. For any subgraph~$H$ of~$G$ and any subset~$I \subseteq V(H)$, the reconnected complement~$H^\ast_{I}$ is a subgraph of~$G^\ast_{I}$.
Moreover, for all~$I \subseteq J \subseteq V(G)$,
\begin{equation}\label{eq:inter}
  (G\vert_{J})^\ast_{I} = G^\ast_{I}\vert_{J \setminus I}.
\end{equation}

\begin{lemma}\label{lemma:rec}
    Let~$I$ be a proper subset of~$V(G)$.
    For each vertex~$v$ in~$V(G) \setminus I$,
    $$
        (G^\ast_I)^\ast_{\{v\}} = G^\ast_{I \cup \{v\}}.
    $$
\end{lemma}
\begin{proof}
    Both graphs have the same vertex set, namely~$V(G) \setminus (I \cup \{v\})$.
    Hence, it remains to verify that their edge sets are equal.
    
    Assume that~$\{a,b\}$ is an edge of~$(G^\ast_I)^\ast_{\{v\}}$.
    Then there exists a path from~$a$ to~$b$ in the induced subgraph~$G^\ast_I\vert_{\{v,a,b\}}$ of~$G^\ast_I$.  
    We distinguish two cases: either~$\{a,b\}$ is an edge of~$G^\ast_I$, or it is not.
    To begin, we consider the case where~$\{a,b\}$ is an edge of~$G^\ast_I$.
    Then there exists a path from~$a$ to~$b$ in~$G\vert_{I \cup \{a,b\}}$, and hence also in~$G\vert_{(I \cup \{v\}) \cup \{a,b\}}$.
    It follows that~$\{a,b\}$ is an edge of~$G^\ast_{I \cup \{v\}}$.
    
    Now, consider the case where~$\{a,b\}$ is not an edge of~$G^\ast_I$.  
    Since~$\{a,b\}$ is an edge of~$(G^\ast_I)^\ast_{\{v\}}$, there must be edges~$\{a,v\}$ and~$\{v,b\}$ in~$G^\ast_I$.  
    This implies the existence of a path from~$a$ to~$v$ in~$G\vert_{I \cup \{a,v\}}$ and a path from~$v$ to~$b$ in~$G\vert_{I \cup \{v,b\}}$.  
    By concatenating these two paths, we conclude that~$\{a,b\}$ is an edge of~$G^\ast_{I \cup \{v\}}$.
    
    Conversely, we assume that~$\{a,b\}$ is an edge of~$G^\ast_{I \cup \{v\}}$.
    Then there exists a path~$P$ from~$a$ to~$b$ in~$G\vert_{(I \cup \{v\}) \cup \{a,b\}}$.
    If the path~$P$ lies entirely within~$G\vert_{I \cup \{a,b\}}$, then~$\{a,b\}$ is an edge of~$G^\ast_I$, and consequently also an edge of~$(G^\ast_I)^\ast_{\{v\}}$.
    If not, the path~$P$ must pass through~$v$.
    In this case, $\{a,v\}$ and~$\{v,b\}$ are edges in~$G^\ast_I$.
    Therefore, $\{a,b\}$ is an edge in~$(G^\ast_I)^\ast_{\{v\}}$.
\end{proof}

A graph class~$\cG$ is said to be \emph{closed under reconnected complements} if, for every~$G \in \cG$ and every~$I \subseteq V(G)$, the reconnected complement~$G^\ast_I$ also belongs to~$\cG$.
When a graph class~$\cG$ is closed under reconnected complements, a graph~$G \in \cG$ is said to be \emph{minimal} in~$\cG$ if there is no proper spanning subgraph~$H$ of~$G$ such that~$H \in \cG$.

A \emph{Hamiltonian graph} is a graph that contains a cycle passing through every vertex exactly once. Such a cycle is called a \emph{Hamiltonian circuit}. 
A \emph{universal vertex} in a graph is a vertex that is adjacent to every other vertex.

\begin{theorem}\label{thm:col}
    The following graph classes are closed under reconnected complements, and their minimal graphs are as follows.
    \begin{enumerate}
        \item\label{thm:col,conn} The class of connected graphs, whose minimal graphs are trees.
        \item\label{thm:col,Ham} The class of Hamiltonian graphs, together with the null, single-vertex, and single-edge graphs. The minimal graphs in this class are the null graph, the single-vertex graph, the single-edge graph, and cycle graphs.
        \item\label{thm:col,uni} The class of graphs with a universal vertex, whose minimal graphs are star graphs.
    \end{enumerate}
    In particular, for a graph~$G$ with a universal vertex~$v$, if~$I \subseteq V(G)$ contains~$v$, then~$G^\ast_I$ is a complete graph, while if~$v \notin I$, then~$G^\ast_I$ is a star graph.
\end{theorem}
\begin{proof}
    From Lemma~\ref{lemma:rec}, to establish the closedness under reconnected complements, it suffices to show that they are closed under graph-collapse at a singleton.
    
    Let~$G$ be a finite simple graph.
    Assume that~$P$ and~$C$ are subgraphs of~$G$ that are, respectively, a path and a cycle, both containing a vertex~$v$.
    Then the neighbors of~$v$ in~$P$ or~$C$ must be adjacent in the reconnected complement~$G^\ast_{\{v\}}$ of~$\{v\}$ in~$G$, which implies that~$P^\ast_{\{v\}}$ and~$C^\ast_{\{v\}}$ remain subgraphs of~$G^\ast_{\{v\}}$.
    Therefore, the closure under graph-collapse is verified for cases~\eqref{thm:col,conn} and~\eqref{thm:col,Ham}.
        
    Now consider the case where~$G$ has a universal vertex~$v$.
    For every pair~$(a, b)$ of distinct vertices in~$V(G) \setminus \{v\}$, there exists a path from~$a$ to~$b$ in the induced subgraph~$G_{\{v,a,b\}}$.
    Hence, in~$G^\ast_{\{v\}}$, all such pairs are adjacent, and~$G^\ast_{\{v\}}$ is a complete graph, as desired.
\end{proof}

\section{Graph $a$-numbers}\label{sec:graph}
Let $G$ be a finite simple graph with vertex set~$V(G)$.
The \emph{signed $a$-number}~$sa(G)$ of~$G$ is defined recursively by setting $sa(\emptyset)=1$, and, for any non-empty graph $G$,
$$
    sa(G) = \left\{
          \begin{array}{ll}
            -\sum_{I\subsetneq V(G)}sa(G\vert_I), & \text{if $G$ has even order;} \\
            0, & \text{otherwise.}
          \end{array}
        \right.
$$
The \emph{$a$-number}~$a(G)$ of~$G$ is defined as~$|sa(G)|$.
In particular, if~$|V(G)|$ is even, then
$$
a(G) = (-1)^{\frac{|V(G)|}{2}}sa(G).
$$
The \emph{$a_i$-number}~$a_i(G)$ of $G$, defined as
$$
    a_i(G):=\sum_{\substack{I\subseteq V(G)\\ |I| =2i}}a(G\vert_I).
$$
The sequence $(a_0(G), a_1(G), \ldots )$ of $G$ is called the \emph{$a$-sequence} of $G$.
The \emph{$b$-number}~$b(G)$ of~$G$ is defined as
$$
    b(G):=\sum_{I\subseteq V(G)}sa(G\vert_I).
$$
Remark that if $G$ consists of $\ell$ connected components $G_1 , \ldots, G_\ell$, then both $a$- and $b$-numbers are multiplicative over components as
$$
    a(G)=a(G_1) \cdots a(G_\ell) \quad \text{ and } \quad     b(G)=b(G_1) \cdots b(G_\ell).
$$

The notions of the $a_i$-number and the $b$-number were introduced in~\cite{Choi-Park2015}, motivated by the canonical real toric variety~$X^\R_G$ associated with~$G$.
Here, the $a_i$-number coincides with the $i$th Betti number of~$X^\R_G$, and the $b$-number with its Euler characteristic.
We now introduce some notation and combinatorial tools to describe key properties of the signed $a$-numbers and the $b$-numbers.

We define the set~$\EC(G)$ as
\begin{equation}\label{cE}
\EC(G) = \{I \subseteq V(G) \colon G\vert_I \text{ has no odd order connected component}\}.
\end{equation}
Assume that $G$ has $2i$ vertices.
We define
$$
\comp(G) = \EC(G) \cup \{V(G)\}.
$$
It is regarded as a graded poset under inclusion.
For~$I \in \comp(G)$ and an integer~$k$ with~$1 \leq k \leq i - \frac{\left\vert I \right\vert}{2}$, let $C_{k}(I;G)$ denote the set of chains
$$
I = I_{0} \subsetneq \cdots \subsetneq I_{k} = V(G),
$$
where~$I_{j} \in \comp(G)$ for all~$1 \leq j \leq k$.
An example of such a poset, arising from a path graph $P_6$ on six vertices, is shown in Figure~\ref{figure:poset}.
\begin{figure}
  \centering
  \begin{tikzpicture}[scale=0.8, every node/.style={
    rectangle, draw, rounded corners=2pt, minimum width=2.2em, minimum height=1.2em, inner sep=2pt, font=\footnotesize}, 
    reverse/.style={yscale=-1}
]
\node (empty) at (9,0.5) {$\emptyset$};

\node (12) at (4,2) {12};
\node (23) at (6.5,2) {23};
\node (34) at (9,2) {34};
\node (45) at (11.5,2) {45};
\node (56) at (14,2) {56};

\node (1234) at (4,4) {1234};
\node (1245) at (6,4) {1245};
\node (2345) at (8,4) {2345};
\node (1256) at (10,4) {1256};
\node (2356) at (12,4) {2356};
\node (3456) at (14,4) {3456};

\node (all) at (9,5.5) {123456};

\draw (all) -- (1234);
\draw (all) -- (1245);
\draw (all) -- (1256);
\draw (all) -- (2345);
\draw (all) -- (2356);
\draw (all) -- (3456);

\draw (1234) -- (12);
\draw (1234) -- (23);
\draw (1234) -- (34);
\draw (1245) -- (12);
\draw (1245) -- (45);
\draw (1256) -- (12);
\draw (1256) -- (56);
\draw (2345) -- (23);
\draw (2345) -- (34);
\draw (2345) -- (45);
\draw (2356) -- (23);
\draw (2356) -- (56);
\draw (3456) -- (34);
\draw (3456) -- (45);
\draw (3456) -- (56);

\foreach \x in {12,23,34,45,56}
    \draw (\x) -- (empty);
\end{tikzpicture}
  \caption{The poset structure of~$\comp(P_6)$}\label{figure:poset}
\end{figure}
For convenience, we set
\begin{equation}\label{eq:zero}
  \left\vert C_k(I;G) \right\vert = 0 \text{ unless } 0 \leq k \leq i - \frac{\left\vert I \right\vert}{2}.
\end{equation}

\begin{proposition}\label{prop:anum}\cite{Choi-Park2015}
    For a simple graph~$G$ on~$2i$ vertices, the signed $a$-number is given by
    $$
    sa(G) = \sum_{k = 0}^{i}(-1)^{k}\left\vert C_{k}(\emptyset;G)\right\vert.
    $$
\end{proposition}

\begin{example}
    Let~$P_6$ be the path graph on~$\{1,\ldots,6\}$ labeled consecutively.
    From the definition of the signed $a$-number in~\eqref{rec}, we compute that~$sa(P_6) = -5$.
    Based on the poset structure shown in Figure~\ref{figure:poset}, we find that
    $$
    |C_0(\emptyset;P_6)| = 0, |C_1(\emptyset;P_6)| = 1, |C_2(\emptyset;P_6)| = 11, \text{ and } |C_3(\emptyset;P_6)| = 15.
    $$
    By Proposition~\ref{prop:anum}, it follows that
    $$
    sa(P_6) = -1+11-15 = -5.
    $$
    Note that this agrees with the value obtained directly from the definition.
\end{example}
Let~$C_6$ be the cycle graph obtained by adding the edge~$\{1,6\}$ to the path graph~$P_6$.
We investigate how this additional edge affects the signed $a$-number.
The set of elements appearing in~$\comp(C_6)$ but not in~$\comp(P_6)$ is
$$
\comp(C_6) \setminus \comp(P_6) = \{\{1,6\},\, \{1,2,3,6\},\, \{1,3,4,6\},\, \{1,4,5,6\}\}.
$$
All chains in~$\comp(C_6)$ can be classified into the following five disjoint types:
\begin{enumerate}
  \item chains that are already contained in~$\comp(P_6)$;
  \item chains that include~$\{1,6\}$;
  \item chains that contain~$\{1,2,3,6\}$ but not~$\{1,6\}$;
  \item chains that contain~$\{1,3,4,6\}$ but not~$\{1,6\}$;
  \item chains that contain~$\{1,4,5,6\}$ but not~$\{1,6\}$.
\end{enumerate}
We now explicitly list all chains of types~(2)–(5) that appear in~$\comp(C_6)$ but not in~$\comp(P_6)$:
\begin{enumerate}
  \item[(2)] $\emptyset \subsetneq \{1,6\} \subsetneq V(C_6)$ in~$C_2(\emptyset;C_6)$,\\[0.4em]
  $\emptyset \subsetneq \{1,6\} \subsetneq I \subsetneq V(C_6)$ in~$C_3(\emptyset;C_6)$ for~$I = \{1,2,5,6\},\{1,2,3,6\},\{1,3,4,6\},\{1,4,5,6\}$.\vspace{.4em}

  \item[(3)] $\emptyset \subsetneq \{1,2,3,6\} \subsetneq V(C_6)$ in~$C_2(\emptyset;C_6)$,\\[0.4em]
  $\emptyset \subsetneq I \subsetneq \{1,2,3,6\} \subsetneq V(C_6)$ in~$C_3(\emptyset;C_6)$ for~$I = \{1,2\},\{2,3\}$.\vspace{.4em}

  \item[(4)] $\emptyset \subsetneq \{1,3,4,6\} \subsetneq V(C_6)$ in~$C_2(\emptyset;C_6)$,\\[0.4em]
  $\emptyset \subsetneq \{3,4\} \subsetneq \{1,3,4,6\} \subsetneq V(C_6)$ in~$C_3(\emptyset;C_6)$.\vspace{.4em}

  \item[(5)] $\emptyset \subsetneq \{1,4,5,6\} \subsetneq V(C_6)$ in~$C_2(\emptyset;C_6)$,\\[0.4em]
 $\emptyset \subsetneq I \subsetneq \{1,4,5,6\} \subsetneq V(C_6)$ in~$C_3(\emptyset;C_6)$ for~$I = \{4,5\},\{5,6\}$.
\end{enumerate}
Refer to Figure~\ref{figure:poset2} for the poset structure of~$\comp(C_6)$, which illustrates the differences from~$\comp(P_6)$.
\begin{figure}
  \centering
  \begin{tikzpicture}[scale=0.9, every node/.style={
    rectangle, draw, rounded corners=2pt, minimum width=2.2em, minimum height=1.2em, inner sep=2pt, font=\footnotesize}, 
    reverse/.style={yscale=-1}
]
\node (empty) at (9,0) {$\emptyset$};

\node (12) at (4,2) {12};
\node (23) at (6.5,2) {23};
\node (34) at (9,2) {34};
\node (45) at (11.5,2) {45};
\node (56) at (14,2) {56};
\node[very thick] (16) at (18,2) {\textbf{16}};

\node (1234) at (4,5) {1234};
\node (1245) at (6,5) {1245};
\node (2345) at (8,5) {2345};
\node (1256) at (10,5) {1256};
\node (2356) at (12,5) {2356};
\node (3456) at (14,5) {3456};
\node[very thick] (1236) at (16,5) {\textbf{1236}};
\node[very thick] (1346) at (18,5) {\textbf{1346}};
\node[very thick] (1456) at (20,5) {\textbf{1456}};

\node (all) at (9,7) {123456};

\draw[draw=gray!50] (all) -- (1234);
\draw[draw=gray!50] (all) -- (1245);
\draw[draw=gray!50] (all) -- (1256);
\draw[draw=gray!50] (all) -- (2345);
\draw[draw=gray!50] (all) -- (2356);
\draw[draw=gray!50] (all) -- (3456);

\draw[draw=gray!50] (1234) -- (12);
\draw[draw=gray!50] (1234) -- (23);
\draw[draw=gray!50] (1234) -- (34);
\draw[draw=gray!50] (1245) -- (12);
\draw[draw=gray!50] (1245) -- (45);
\draw[draw=gray!50] (1256) -- (12);
\draw[draw=gray!50] (1256) -- (56);
\draw[draw=gray!50] (2345) -- (23);
\draw[draw=gray!50] (2345) -- (34);
\draw[draw=gray!50] (2345) -- (45);
\draw[draw=gray!50] (2356) -- (23);
\draw[draw=gray!50] (2356) -- (56);
\draw[draw=gray!50] (3456) -- (34);
\draw[draw=gray!50] (3456) -- (45);
\draw[draw=gray!50] (3456) -- (56);
\draw[very thick] (1256) -- (16);
\draw[very thick] (1236) -- (12);
\draw[very thick] (1236) -- (23);
\draw[very thick] (1346) -- (34);
\draw[very thick] (1456) -- (45);
\draw[very thick] (1456) -- (56);
\draw[very thick] (all) -- (1236) -- (16) -- (empty);
\draw[very thick] (all) -- (1346) -- (16) -- (empty);
\draw[very thick] (all) -- (1456) -- (16) -- (empty);

\foreach \x in {12,23,34,45,56}
    \draw[draw=gray!50] (\x) -- (empty);
\end{tikzpicture}
  \caption{The poset structure of $\comp(C_6)$ relative to $\comp(P_6)$}\label{figure:poset2}
\end{figure}
From the above, we observe that
$$
|C_{2}(\emptyset;C_6) \setminus C_{2}(\emptyset;P_6)| =4 \quad \text{and} \quad |C_{3}(\emptyset;C_6) \setminus C_{3}(\emptyset;P_6)| = 9.
$$
By Proposition~\ref{prop:anum}, the signed $a$-number of~$C_6$ is given by
$$
sa(C_6) = sa(P_6) + \sum_{k = 0}^{i}(-1)^{3}\left\vert C_{k}(\emptyset;C_6) \setminus C_k(\emptyset;P_6)\right\vert = -5 +4 -9 = -10.
$$
The purpose of this paper is to demonstrate that such a decomposition is not a phenomenon restricted to this specific case, but rather a general principle that holds for arbitrary finite simple graphs; see Section~\ref{sec:decom}.

Now, we introduce Proposition~\ref{prop:oddeven}, which determines the sign of the $b$-numbers.
\begin{proposition}\label{prop:oddeven}
Let~$G$ be a finite simple graph with vertex set $V(G)$.
\begin{enumerate}
  \item\label{prop:oddeven1} If~$\left\vert V(G) \right\vert$ is odd, then
  \[
    \left\vert b(G) \right\vert = (-1)^{\frac{\left\vert V(G) \right\vert -1}{2}} b(G).
  \]
  
  \item\label{prop:oddeven2} If~$G$ has a connected component of even order, then
  \[
    b(G) = 0.
  \]
  
  \item\label{prop:oddeven3} Assume that~$\left\vert V(G) \right\vert = 2i$ is even, and~$G$ has no connected component of even order.
      If the number of connected components of~$G$ is $2j$, then
  \[
    \left\vert b(G) \right\vert = (-1)^{i-j} \cdot b(G).
  \]
\end{enumerate}
\end{proposition}

\begin{proof}
    The first two cases were proved in~\cite{Choi-Park2015}.
    We prove the third case.
    Let~$G_1, \ldots, G_{2j}$ be the connected components of~$G$, where~$\left\vert V(G_i) \right\vert = 2k_i + 1$ for some non-negative integers~$k_i$.
    Since~$G$ has~$2i$ vertices, we have
    $$
    k_1+ \cdots + k_{2j} + j= i.
    $$
    By \eqref{prop:oddeven1}, for each~$1 \leq i \leq 2j$, we know that
    $$
    \left\vert b(G_i) \right\vert = (-1)^{k_i} \cdot b(G_i).
    $$
    Therefore, we obtain
    \begin{align*}
      \left\vert b(G) \right\vert & =\left\vert b(G_1) \cdot b(G_2) \cdots b(G_{2j}) \right\vert \\
       & =\left\vert b(G_1) \right\vert \cdot \left\vert b(G_2) \right\vert  \cdots \left\vert b(G_{2j}) \right\vert \\
       & = (-1)^{k_1+\cdots+k_{2j}} \cdot b(G) \\
       & = (-1)^{i-j} \cdot b(G),
    \end{align*}
    as desired.
    \end{proof}

\section{A decomposition formula}\label{sec:decom}

Let $G$ be a finite simple graph.
For each pair~$\edge$ of vertices of~$G$, let~$G + \edge$ denote the graph obtained from~$G$ by adding~$\edge$ as an edge.
Note that~$G = G +\edge$ if~$\edge$ is already an edge of~$G$.

Recall the definition of~$\EC(G)$ from~\eqref{cE}, and we now describe some of its properties.

\begin{proposition}\label{prop:conn}
    For a finite simple graph~$G$, let~$J \in \EC(G)$, and~$\edge$ a pair of vertices of~$G$.
    The following statements hold.
    \begin{enumerate}
      \item\label{prop:conn1} For~$I \subseteq J$, $I \in \EC(G)$ if and only if~$I \in \EC(G\vert_J)$. 
      \item\label{prop:conn2} For~$J \subseteq I$, $I \in \EC(G)$ if and only if~$I \setminus J \in \EC(G^\ast_J)$.
      \item\label{prop:conn3} $\EC(G) \subseteq \EC(G+\edge)$.
    Moreover, $I \in \EC(G+\edge) \setminus \EC(G)$ implies that $\edge \subseteq I$.
    \end{enumerate}
\end{proposition}
\begin{proof}
The statements~\eqref{prop:conn1} and~\eqref{prop:conn3} are straightforward to verify.
We now consider case~\eqref{prop:conn2}.
By Lemma~\ref{lemma:rec}, we may assume without loss of generality that~$|J|=2$, in which case~$J$ is regarded as an edge of~$G$.

Assume~$I \in \EC(G)$ and~$J \subseteq I$.
Then there exists a subset~$I_0 \subseteq I$ such that~$G\vert_{I_0}$ is a connected component of~$G\vert_{I}$ and contains~$J$.
By~\eqref{thm:col,conn} in Theorem~\ref{thm:col}, the reconnected complement~$(G\vert_{I_0})^\ast_J$ is connected.
Moreover, since~$(G\vert_{I_0})^\ast_J$ is also a connected component of~$G^\ast_J$ with even order, it follows that~$I \setminus J \in \EC(G^\ast_J)$.

Conversely, assume that~$I \setminus J \in \EC(G^\ast_J)$.
If~$G\vert_{J}$ is an isolated edge of~$G$, then we observe that~$G^\ast_J = G\vert_{V(G) \setminus J}$, where~$V(G)$ is the set of vertices of~$G$.
Then it is immediate that~$I \in \EC(G)$.
If~$G\vert_J$ is not isolated in~$G$, then there exists a connected component~$G^\ast_J\vert_{I_0}$ of~$G^\ast_J$ such that~$G\vert_{I_0 \cup J}$ forms connected component of~$G$.
Since~$|I_0 \cup J|$ is even, $I \in \EC(G)$.
This completes the proof.
\end{proof}

Assume that~$G$ is a simple graph with~$2i$ vertices.
\begin{lemma}\label{lemma:collap}
    Given~$J \in \comp(G)$,
    $$
    \left\vert C_k(J;G) \right\vert = \left\vert C_k(\emptyset;G^\ast_J) \right\vert
    $$
    for each~$0 \leq k \leq i - \frac{\left\vert J \right\vert}{2}$.
\end{lemma}

\begin{proof}
    Since the statement is immediate when~$k = 0$, we assume~$k \geq 1$.
    Let $\cC \in C_k(J;G)$ be a chain of the form
    $$
    J= I_{0} \subsetneq \cdots \subsetneq I_{k} = V(G).
    $$
    Define $\varphi_J(\cC)$ to be a chain
    $$
    \emptyset = (I_0 \setminus J) \subsetneq \cdots \subsetneq (I_k \setminus J) = V(G^\ast_J).
    $$
    For each~$1 \leq \ell \leq k$, we have $I_\ell \in \comp(G)$, and by~\eqref{prop:conn2}~in~Proposition~\ref{prop:conn}, it follows that~$I_\ell \setminus J \in \comp(G^\ast_J)$.
    We obtain that~$\varphi_J(\cC) \in C_k(\emptyset; G^\ast_J)$, and, hence, $\varphi_J$ is regarded as a map from~$C_k(J;G)$ to~$C_k(\emptyset;G^\ast_J)$.
    
    Conversely, let~$\cC \in C_k(\emptyset;G^\ast_J)$ be a chain of the form
    $$
    \emptyset= I^\ast_{0} \subsetneq \cdots \subsetneq I^\ast_{k} = V(G^\ast_J).
    $$
    Define~$\psi_J(\cC)$ to be
    $$
    J = J \cup I^\ast_{0} \subsetneq \cdots \subsetneq J \cup I^\ast_{k} = V(G).
    $$
    Since each~$J \cup I^\ast_\ell$ belongs to~$\comp(G)$ by Proposition~\ref{prop:conn}\eqref{prop:conn2}, it follows that~$\psi_J(\cC) \in C_k(J;G)$, and thus~$\psi_J$ defines a map from~$C_k(\emptyset; G^\ast_J)$ to~$C_k(J;G)$.
    
    It is straightforward to check that $\varphi_J$ and $\psi_J$ are inverses of each other, establishing a one-to-one correspondence between $C_k(J;G)$ and $C_k(\emptyset; G^\ast_J)$.
\end{proof}

Let a pair~$\edge$ of vertices of~$G$ be given.
Consider a chain~$\cC \in C_k(\emptyset;G+\edge)$ given by
$$
\emptyset = I_0 \subsetneq \cdots \subsetneq I_k = V(G+\edge).
$$
Assume that there is an index~$0 \leq \ell \leq k$ such that~$I_\ell \notin \EC(G)$.
Since~$\emptyset \in \EC(G)$, any such index~$\ell$ should be at least~$1$.
We choose the smallest such index~$\ell$, and we call~$\cC$ \emph{$\edge$-singular} at~$I_\ell$.
If no such index exists, then~$\cC$ is said to be \emph{$\edge$-regular}.

\begin{lemma}\label{lemma:split}
    Let~$G$ be a simple graph with~$2i$ vertices, where~$ i \geq 1$, and $\edge$ a pair of vertices of~$G$.
    If~$J \in \EC(G+\edge) \setminus \EC(G)$, then the number of $\edge$-singular chains~$\cC \in C_k(\emptyset;G+\edge)$ at~$J$ is given by
    $$
    \sum_{p+q = k-1}\Bigl(\bigl|  C_p(J; G + \edge) \bigl| \cdot \sum_{J' \in \EC(G\vert_J)}\left\vert C_q(\emptyset; G\vert_{J'})\right\vert\Bigl),
    $$
    where~$1 \leq k \leq i$.
\end{lemma}

\begin{proof}
    Let~$\cC \in C_k(\emptyset;G+\edge)$ be an $\edge$-singular chain at~$J$, where~$k \geq 1$, given by:
    $$
    \emptyset = I_0 \subsetneq \cdots \subsetneq I_k = V(G+\edge).
    $$
    Then there exists an index~$1 \leq \ell \leq k$ such that~$I_\ell = J$.
    We denote~$I_{\ell-1}$ by~$J'$.
    By Proposition~\ref{prop:conn}~\eqref{prop:conn1}, $J' \in \EC(G)$ and~$J' \subseteq J$, and hence~$J' \in \EC(G\vert_J)$ as well.
    Thus, $\cC$ splits into two chains:
    $$
    \emptyset = I_0 \subsetneq \cdots \subsetneq I_{\ell-1} = J' \quad \text{and} \quad J = I_{\ell} \subsetneq \cdots \subsetneq I_k = V(G+\edge).
    $$
    The former is an element of~$C_{\ell-1}(\emptyset;G\vert_{J'})$ and the latter is an element of~$C_{k-\ell}(J;G+\edge)$.
    
    Now, choose $J' \subsetneq J$ with~$J' \in \EC(G\vert_J)$.
    Let chains in~$C_p(J;G+e)$ and in~$C_q(\emptyset;G\vert_{J'})$ be defined as
    $$
    J = J_0 \subsetneq \cdots \subsetneq J_p = V(G+\edge) \quad \text{and} \quad \emptyset = J'_0 \subsetneq \cdots \subsetneq J'_q = V(G\vert_{J'}),
    $$
    respectively, where~$p, q \geq 0$ and $p+q = k-1$.
    We define $\cC$ as the chain obtained by combining the two chains above as
    $$
    \emptyset = J'_0 \subsetneq \cdots \subsetneq J'_q = J' \subsetneq J = J_0 \subsetneq \cdots \subsetneq J_p =V(G+\edge).
    $$
    By~\eqref{prop:conn1} and~\eqref{prop:conn3}~of~Proposition~\ref{prop:conn}, it follows that, for each~$0 \leq \ell \leq q$,
    $$
    J'_{\ell} \in \EC(G\vert_{J'}) \subseteq \EC(G) \subseteq \EC(G+\edge).
    $$ 
    Moreover, since~$J \in \EC(G+\edge)\setminus \EC(G)$, we conclude that $\cC$ is $\edge$-singular at~$J$.
    
    Finally, the splitting and combining described above yield a one-to-one correspondence between the set of~$\edge$-singular chains in~$C_k(\emptyset;G+\edge)$ at~$J$ and the set
    $$
    \bigsqcup_{p+q = k-1}\Bigl( C_p(J; (G + \edge))  \times \bigsqcup_{J' \in \comp(G\vert_J)} C_q(\emptyset; G\vert_{J'})\Bigl).
    $$
    This completes the proof.
    \end{proof}

\begin{lemma}\label{lemma:decom}
    Let $G$ be a simple graph on~$2i$ vertices.
    For a pair~$\edge$ of vertices of~$G$,
    $$
    a(G+\edge) = a(G) + \sum_{J \in \EC(G+\edge) \setminus \EC(G)} \left\vert b(G\vert_J) \right\vert \cdot a((G+\edge)^\ast_J).
    $$
\end{lemma}

\begin{proof}
    For each~$0 \leq k \leq i$, let~$RC_k$ denote the set of~$\edge$-regular chains in~$C_k(\emptyset;G+\edge)$, and~$SC^J_k$ the set of $\edge$-singular chains at $J \in \EC(G+\edge) \setminus \EC(G)$ in~$C_k(\emptyset;G+\edge)$.
    It is immediate that
    $$
    C_k(\emptyset;G+\edge) = RC_k \sqcup \Bigl(\bigsqcup_{J \in \EC(G+\edge) \setminus \EC(G)} SC_k^J \Bigl).
    $$
    
    Let~$1 \leq k \leq i$ be given. 
    It can be readily verified that~$RC_k= C_k(\emptyset;G)$.
    Now let us consider~$J \in \EC(G+\edge)\setminus \EC(G)$ with~$k \leq j := \frac{\left\vert J \right\vert}{2} \leq i$.
    Since it is clear that~$G\vert_J$ has exactly two components of odd order, by Proposition~\ref{prop:oddeven}, we have
    \begin{equation}\label{eq:bGJ}
      \left\vert b(G\vert_J) \right\vert = (-1)^{j-1} \cdot b(G\vert_J).
    \end{equation}
    Furthermore, it follows that
    \begin{align*}
      \left\vert SC^J_k \right\vert & =  \sum_{p+q = k-1}\Bigl(\bigl| C_p(J; G + \edge) \bigl| \cdot \sum_{J' \in \EC(G\vert_J)}\left\vert C_q(\emptyset; G\vert_{J'})\right\vert\Bigl) \text{ (by Lemma~\ref{lemma:split})} \\
       & = \sum_{p+q = k-1}\Bigl(\bigl| C_p(\emptyset; (G + \edge)^\ast_J) \bigl| \cdot \sum_{J' \in \EC(G\vert_J)}\left\vert C_q(\emptyset; G\vert_{J'})\right\vert\Bigl) \text{ (by Lemma~\ref{lemma:collap})}.
    \end{align*}
    Using the convention stated in~\eqref{eq:zero}, we derive the following equality:
    \begin{align*}
     \sum_{k=0}^{i}(-1)^{k}\left\vert SC^J_k\right\vert & = -\sum_{p=0}^{i}\sum_{q=0}^{i}\Bigl((-1)^{p+q}\bigl| C_p(\emptyset; (G + \edge)^\ast_J) \bigl| \cdot \sum_{J' \in \EC(G\vert_J)} \left\vert C_q(\emptyset; G\vert_{J'})\right\vert\Bigl) \\ 
       & = -\sum_{p=0}^{i}\Bigl((-1)^p\bigl| C_p(\emptyset; (G + \edge)^\ast_J) \bigl| \cdot \sum_{J' \in \EC(G\vert_J)} \sum_{q=0}^{i}(-1)^q\left\vert C_q(\emptyset; G\vert_{J'})\right\vert\Bigl) \\
       & = -\sum_{p=0}^{i}\Bigl((-1)^p\bigl| C_p(\emptyset; (G + \edge)^\ast_J) \bigl| \cdot \sum_{J' \in \EC(G\vert_J)} sa(G\vert_{J'})\Bigl) \quad \text{(by Proposition~\ref{prop:anum})}\\
       & = -b(G\vert_J)\sum_{p=0}^{i}(-1)^p\bigl| C_p(\emptyset; (G + \edge)^\ast_J) \bigl| \ \quad \bigl(\because b(G\vert_J) = \sum_{J' \in \EC(G\vert_J)} sa(G\vert_{J'})\bigl)\\
       &= -b(G\vert_J) \cdot sa((G + \edge)^\ast_J) \quad \text{(by Proposition~\ref{prop:anum})}\\
       &= (-1)^{j-1} \cdot b(G\vert_J) \cdot (-1)^{j} \cdot sa((G + \edge)^\ast_J) \\
       &= \left\vert b(G\vert_J) \right\vert \cdot a((G + \edge)^\ast_J) \quad \text{(by \eqref{eq:bGJ})}.
    \end{align*}
    We conclude that
    \begin{align*}
      a(G+\edge) & = \sum_{k = 0}^{i}(-1)^{i+k}\left\vert C_{k}(\emptyset;G+\edge)\right\vert
     \\
     & = \sum_{k = 0}^{i}(-1)^{i+k} \Bigl( \left\vert RC_{k}\right\vert + \sum_{J \in \EC(G+\edge) \setminus \EC(G)}\left\vert SC^J_k\right\vert \Bigl)
     \\
       & = \sum_{k = 0}^{i}(-1)^{i+k}\left\vert C_{k}(\emptyset;G)\right\vert + \sum_{J \in \EC(G+\edge) \setminus \EC(G)}\sum_{k=0}^{i}(-1)^{i+k}\left\vert SC^J_k\right\vert
     \\
     & = a(G) + \sum_{J \in \EC(G+\edge) \setminus \EC(G)} \left\vert b(G\vert_J) \right\vert \cdot a((G + \edge)^\ast_J).
    \end{align*}
\end{proof}

\begin{theorem}\label{thm:decom}
    Let $G$ be a finite simple graphs with the vertex set~$V(G)$, and $\edge$ a pair of vertices of~$G$.
For each $i \geq 0$, 
    \begin{equation} \label{eqn:main_decomp}
    a_i(G+\edge) - a_i(G) = \sum_{J \in \EC(G+\edge) \setminus \EC(G)}\left\vert  b(G\vert_J) \right\vert \cdot a_{i- \frac{|J|}{2}}((G +\edge)^\ast_{J}).
    \end{equation}
\end{theorem}
\begin{proof}
    To begin, since~$V(G) = V(G+\edge)$, we have
    $$a_i(G+\edge) = \sum_{\substack{I\subseteq V(G)\\ |I| =2i}}a((G+\edge)\vert_I).
    $$
    We now divide the sum according to whether~$I$ contains~$\edge$ or not; that is,
    $$
    a_i(G+\edge) = \sum_{\substack{\edge \not\subset I\subseteq V(G)\\ |I| =2i}} a((G+\edge)\vert_I) + \sum_{\substack{\edge \subseteq I\subseteq V(G)\\ |I| =2i}} a((G+\edge)\vert_I).
    $$
    For each~$I \in \EC(G + \edge)$, note that
    $$
    (G+\edge)\vert_I = \begin{cases}
      G\vert_I +\edge, & \mbox{if } \edge \in I \\
      G\vert_I, & \mbox{otherwise}.
    \end{cases}
    $$
    It follows that
    $$
    a_i(G+\edge) = \sum_{\substack{\edge \not\subset I\subseteq V(G)\\ |I| =2i}} a(G\vert_I) + \sum_{\substack{\edge \subseteq I\subseteq V(G)\\ |I| =2i}} a(G\vert_I+\edge).
    $$
    We then apply Lemma~\ref{lemma:decom} to~$G\vert_I +\edge$, which confirms that
    \begin{align*}
      a_i(G+\edge) & = \sum_{\substack{\edge \not\subset I\subseteq V(G)\\ |I| =2i}} a(G\vert_I) + \sum_{\substack{\edge \subseteq I\subseteq V(G)\\ |I| =2i}} \Big(a(G\vert_I) + \sum_{J \in \EC((G+\edge)\vert_I) \setminus \EC(G\vert_I)} \left\vert b(G\vert_J) \right\vert a\big(((G+\edge)\vert_I)^\ast_J\big)\Big) \\
       & = a_i(G) + \sum_{\substack{\edge \subseteq I\subseteq V(G)\\ |I| =2i}} \sum_{\substack{J \in \EC(G+\edge) \setminus \EC(G) \\ J \subseteq I}} \left\vert b(G\vert_J) \right\vert a\big(((G+\edge)\vert_I)^\ast_J\big) \quad \text{ (by Proposition~\ref{prop:conn}~\eqref{prop:conn1})}.
    \end{align*}
    From Proposition~\ref{prop:conn}~\eqref{prop:conn3}, for each~$J \in \EC(G+\edge) \setminus \EC(G)$, we have $\edge \subseteq J$.
    It follows that interchanging the order of summation in the second term on the right-hand side yields
    $$
       a_i(G+\edge) = a_i(G) + \sum_{J \in \EC(G+\edge) \setminus \EC(G)}\Big( \left\vert b(G\vert_J) \right\vert \sum_{\substack{J \subseteq I\\ |I| =2i}} a\big(((G+\edge)\vert_I)^\ast_J\big)\Big).
    $$
    By replacing $I' =I \setminus J$ with~\eqref{eq:inter}, we obtain that
    \begin{align*}
      a_i(G+\edge) & = a_i(G) + \sum_{J \in \EC(G+\edge) \setminus \EC(G)}\Big( \left\vert b(G\vert_J) \right\vert \sum_{\substack{I' \in V((G+\edge)^\ast_J)\\ |I'| =2i-\left\vert J \right\vert}} a\big(((G+\edge)^\ast_J)_{I'})\big)\Big) \\
       & = a_i(G) + \sum_{J \in \EC(G+\edge) \setminus \EC(G)} \left\vert b(G \vert_J) \right\vert \cdot a_{i-\frac{|J|}{2}}((G +\edge)^\ast_J),
    \end{align*}
as desired.
\end{proof}

\begin{example} \label{example:path}
Let
$$
    C(n,k) = \binom{n+k}{k} - \binom{n+k}{k-1}
           = \frac{n-k+1}{n+1}\binom{n+k}{k}.
$$
It is well known that $C(n,k)$ counts the strings with $n$ letters $X$ and $k$ letters $Y$ such that no initial segment of the string contains more $Y$'s than $X$'s (\cite{oeis} A008315).
These numbers form the entries of the Catalan triangle, and $C(n,n)$ is the $n$th Catalan number, denoted by $\cC_n$.
For the path graph $P_n$, it was shown in~\cite{Choi-Park2015} that, for $0\le i\le \lfloor (n+1)/2\rfloor$,
$$
    a_i(P_n) = C(n-i,i).
$$
In this example, we reprove this formula using the decomposition formula~\eqref{eqn:main_decomp}.

We regard $P_n$ as the graph on $[n+1]=\{1,\ldots,n+1\}$ with edges $\{i,i+1\}$ for $i=1,\ldots,n-1$, so that its $a_i$-numbers agree with those of the usual path graph on $n$ vertices. 
Let $e=\{n,n+1\}$ and observe that $P_n+e = P_{n+1}$.
We now describe the subsets $J$ appearing as non-zero terms on the right-hand side in \eqref{eqn:main_decomp}. 
By Proposition~\ref{prop:oddeven}~(2), the admissible subsets $J$ are exactly those subsets of $[n+1]$ that contain $n+1$ and consist of consecutive vertices of even cardinality. 
Writing $|J|=2j$, such a subset is of the form
$$
    J = \{n+2-2j,\ldots,n+1\}.
$$
In this case, one easily checks that
$$
    P_n|_J \cong P_{2j-1} \quad\text{ and }\quad (P_{n+1})^\ast_J \cong P_{n+1-2j}.
$$

We proceed by induction on $n$. Assume that for all $m \leq n$ and all $i$,
$$
    a_i(P_m) = C(m-i,i).
$$
Then, for $2j-1 \leq n$, we have
$$
    b(P_{2j-1}) = \cC_{j-1}.
$$
Using the induction hypothesis, the decomposition formula yields
\begin{align*}
    a_i(P_{n+1}) &= C(n-i,i) +\sum_{j=1}^{i} \cC_{j-1} C(n+1-i-j, i-j) \\
  &  = C(n-i,i) + \sum_{j=0}^{i} \cC_{j} C(n-i-j, i-j-1) \\
  &  = C(n-i,i) + C(n+1-i, i-1) \\
  &  = C(n+1-i,i), 
\end{align*}
which completes the proof.
\end{example}

By Theorem~\ref{thm:decom}, the monotonicity of the $a_i$-numbers and the $a$-numbers are established, as shown in Corollary~\ref{cor:mon}.
\begin{corollary}\label{cor:mon}
    Let~$G$ be a finite simple graph and $H$ a subgraph of $G$.
    Then, for all $i\geq 0$,
    $$
        a_i(H) \leq a_i(G).
    $$
    In particular, if $H$ is a spanning subgraph of $G$, then
    $$
        a(H) \leq a(G).
    $$

\end{corollary}
\begin{proof}
  If~$H$ is a spanning subgraph of~$G$, the corollary is clear.
  Otherwise, we add isolated vertices to~$H$  so that it has the same number of vertices as~$G$.
  Note that this modification does not change $a_i(H)$ for all $i\geq 0$. 
  Therefore, the corollary follows immediately from Theorem~\ref{thm:decom}.   
\end{proof}

\section{Upper and Lower bounds of $a_i$-numbers} \label{sec:bounds}

In the previous section, we discussed the monotonicity of the $a_i$-numbers. 
In this section, we use this property to derive bounds for the $a_i$-numbers and thereby prove Corollary~\ref{coro:upperlower}. 
By Corollary~\ref{cor:mon}, the $a_i$-number of a graph $G$ is sandwiched between the $a_i$-number of a spanning tree of $G$ and that of the complete graph on the same number of vertices. 
Therefore, to complete the proof of Corollary~\ref{coro:upperlower}, it suffices to show that for an arbitrary tree $T$, its $a_i$-number lies between those of the path graph and the star graph with the same number of vertices.

Throughout this section, the path graph and the star graph on $n$ vertices are denoted by~$P_{n}$ and  $K_{1,n-1}$, respectively.
It is known that for $0\le i\le \lfloor (n+1)/2\rfloor$,
$$
    a_i(P_n)=\frac{n-2i+1}{n-i+1}\binom{n}{i}, \quad \text{ and } \quad  a_i(K_{1,n-1})=\binom{n-1}{2i-1}A_{2i-1},
$$
where $A_k$ denotes the $k$th Euler zigzag number.
See Example~\ref{example:path} and \cite{Choi-Park2015} for details.

\begin{theorem} \label{thm:bounds_a_i_numbers}
    Let~$T_n$ be a tree on~$n$ vertices.
    Then, for each non-negative integer~$i$,
    \begin{equation}\label{eq:treeupperlower}
        a_i(P_{n})\le a_i(T_n) \le a_i(K_{1,n-1}).
    \end{equation}
\end{theorem}

To prove the above theorem, we prepare the following lemma.

\begin{lemma}\label{lem:bounds_b_tree_odd}
For any tree $T_{n}$ on $n$ vertices,
$$
    |b(P_{n})|\le |b(T_{n})|\le |b(K_{1,n-1})|.
$$
\end{lemma}
\begin{proof}
Let $G$ be a graph on $n$ vertices, and let $X_G^{\R}$ be the real toric variety associated with $G$.
By definition, $b(G)$ is the alternating sum of the $a_i$-numbers, and $a_i(G)=\beta_i(X_G^{\R};\Q)$, where
$$
    \beta_i(X;\F)=\dim_{\F} H_i(X;\F)
$$
denotes the $i$th Betti number of a space $X$ with coefficients in a field $\F$.
Thus
\begin{align*}
b(G) &=\sum_{i\ge 0}(-1)^i \beta_i(X_G^{\R};\Q)=\chi(X_G^{\R}) \\
     &=\sum_{j\ge 0}(-1)^j \beta_j(X_G^{\R};\Z_2),
\end{align*}
where $\chi(X)$ denotes the Euler characteristic of a space $X$.

Assume that $n=2k+1$.
Let $(h_0,\ldots,h_{2k})$ be the $h$-vector of the graph associahedron $P_G$, and let
$\gamma=(\gamma_0,\ldots,\gamma_k)$ be the \emph{gamma vector} of $P_G$ defined by
$$
    \sum_{i=0}^{2k} h_i t^i=\sum_{i=0}^{k}\gamma_i  t^i(1+t)^{2k-2i}.
$$
By \cite{Davis-Januszkiewicz1991}, we have $\beta_j(X_G^{\R};\Z_2)=h_j$ for all $j$, hence
$$
    b(G)=\sum_{j=0}^{2k}(-1)^j h_j.
$$
Evaluating the defining identity at $t=-1$ gives $\sum_{j=0}^{2k}(-1)^j h_j=(-1)^k\gamma_k$, hence
$$
    |b(G)|=\gamma_k.
$$

By \cite{Bukhshtaber-Volodin2011}, for every $i$,
$$
    \gamma_i(P_{P_n})\le \gamma_i(P_{T_n})\le \gamma_i(P_{K_{1,n-1}}).
$$
Applying this with $i=k$ and using $|b(G)|=\gamma_k$ proves the claim for $n=2k+1$.
If $n$ is even, then $b(G)=0$ by Proposition~\ref{prop:oddeven}, so the inequality holds.
\end{proof}

\begin{proof}[Proof of Theorem~\ref{thm:bounds_a_i_numbers}]
    We will prove the theorem by mathematical induction on~$n$.
    It is straightforward to verify that~\eqref{eq:treeupperlower} holds for all $n \le 4$.
Assume that~\eqref{eq:treeupperlower} holds for all~$n \le N$.
A tree~$T_{N+1}$ has at least two leaves.
Let~$v$ be one such leaf, and let~$u$ denote the vertex adjacent to~$v$ in~$T_{N+1}$.
Define~$F_N$ to be the forest obtained from~$T_{N+1}$ by deleting the edge~$e=(u,v)$, \ie
$$
T_{N+1}=F_N+e.
$$
Note that
$$
a_i(P_{N}) \leq a_i(F_N) \leq a_i(K_{1,N-1}).
$$

Applying Theorem~\ref{thm:decom}, we have
\begin{equation}\label{eq:proof_in_upplow}
  a_i(T_{N+1}) - a_i(F_N) = \sum_{J \in \EC(T_{N+1}) \setminus \EC(F_{N})}\left\vert  b(F_N\vert_J) \right\vert \cdot a_{i- \frac{|J|}{2}}((T_{N+1})^\ast_{J}).
\end{equation}
The difference $\EC(T_{N+1})\setminus \EC(F_N)$ consists precisely of those sets of the form $J = J'\cup \{v\}$, where~$F_{N}\vert_{J'}$ is a connected subgraph of odd order.
For any~$J \in \EC(T_{N+1}) \setminus \EC(F_N)$, we have the following observations.
\begin{enumerate}
  \item $b(F_N\vert_J) = b(P_{|J|})$ and $ (T_{N+1})^\ast_J = P_{N+1-J}$ when~$T_{N+1}$ is a path graph.
  \item $b(F_N\vert_J) = b(K_{|J|-1,1})$ and $ (T_{N+1})^\ast_J = K_{N+1-J}$, when~$T_{N+1}$ is a star graph.
\end{enumerate}

Since~$T_{N+1}$ is connected, for each~$k \geq 0$, there is at least one such~$J'$ with~$|J'| = 2k+1$.
For such subsets~$J$ appearing in~\eqref{eq:proof_in_upplow}, the number of admissible choices is minimized when~$T_{N+1}$ is a path graph and maximized when~$T_{N+1}$ is a star graph for each even cardinality of~$J$.
Hence, the lower and upper bounds of the right-hand side of~\eqref{eq:proof_in_upplow} are attained when~$T_{N+1}$ is a path graph and a star graph, respectively.
Combining this with Lemma~\ref{lem:bounds_b_tree_odd} and the induction hypothesis, we complete the proof of the theorem.
\end{proof}

\section{Unimodality of $a$-sequences}\label{sec:unimodal}

In this section, for each simple graph~$G$, we discuss the unimodality of the $a$-sequence
$$
(a_0(G),a_1(G),a_2(G),\ldots)
$$
in $i$ of~$G$.
Note that this sequence is eventually zero.

We say that a sequence is \emph{penultimately increasing} if it satisfies the following two conditions:
\begin{enumerate}
  \item it is eventually zero, and
  \item it increases up to its second-to-last nonzero term.
\end{enumerate}
In this case, the sequence is trivially unimodal.

\begin{theorem}\label{thm:unimodal}
    Let~$\cG$ be a graph class closed under reconnected complement.
    If the $a$-sequence of each minimal graph in~$\cG$ is penultimately increasing, then the same holds for every graph in~$\cG$.
\end{theorem}

\begin{proof}
    Let $G \in \cG$ be a simple graph on~$m$ vertices.
    It is immediate that the $a$-sequence of~$G$ is penultimately increasing when~$m \leq 2$.
    We proceed by induction on~$m$.
    Assume that for some integer~$N \geq 2$, the $a$-sequence of every graph in~$\cG$ with at most~$N$ vertices is penultimately increasing.
    
    Now let~$G \in \cG$ be a graph on~$N+1$ vertices.
    If~$G$ is minimal in~$\cG$, then by assumption, its $a$-sequence is penultimately increasing.
    Otherwise, there exists a graph~$G' \in \cG$ with~$N+1$ vertices and an edge~$\edge$ such that~$G = G' + \edge$, where~$\edge$ is not an edge of~$G'$.
    Suppose that the $a$-sequence of~$G'$ is penultimately increasing: that is,
    $$
    a_0(G') \leq a_1(G') \leq \cdots \leq a_{\lfloor \frac{N-1}{2}\rfloor}(G').
    $$
    For each~$J \in \EC(G)$ with~$\frac{|J|}{2} = j$, consider the induced graph~$G^\ast_J = (G' + \edge)^\ast_J$.
    Since~$G^*_J$ has fewer than~$N+1$ vertices, the inductive hypothesis applies, and thus
    $$
    a_0\big( G^\ast_{J} \big) \leq \cdots \leq a_{\lfloor \frac{N-1}{2}\rfloor-j}\big( G^\ast_{J} \big).
    $$
    Then, by Theorem~\ref{thm:decom}, the $a$-sequence of~$G$ is also penultimately increasing, completing the inductive step.
    Therefore, by induction, the $a$-sequence of every graph~$G \in \cG$ is penultimately increasing.
\end{proof}

\begin{corollary}\label{cor:Ham}
  If~$G$ is a Hamiltonian graph, then the $a$-sequence is penultimately increasing, and hence is unimodal.
\end{corollary}
\begin{proof}
    Consider the graph class consisting of Hamiltonian graphs together with the single-edge graph.
    Observe that the cycle graph of order~$n \geq 3$ is the unique minimal graph with $n$ vertices in the class.
    It is known~\cite{Choi-Park2015} that the $a$-sequence of each cycle graph coincides with the half Pascal triangle, and in particular, it is penultimately increasing.
    Therefore, by Theorem~\ref{thm:col}~\eqref{thm:col,Ham} and Theorem\ref{thm:unimodal}, we conclude this corollary.
\end{proof}

Now, we consider the star graph~$K_{1,n-1}$.
For each~$0 \leq i \leq \lfloor \frac{n}{2} \rfloor$, the $a_i$-number of a star graph is computed by
\begin{equation}\label{eq:star}
  a_i(K_{1,n-1}) = \binom{n-1}{2i-2}A_{2i-1}, 
\end{equation}
where $A_k$ denotes the Euler zig-zag number (\cite{oeis} A000111).
In particular, $A_{2n-1}$ is known as the tangent number (\cite{oeis} A000182).
For more details, see~\cite{Choi-Park2015}.

Recall that
\begin{equation}\label{eq:AB}
  A_{2i-1} = \frac{2^{2i} (2^{2i} - 1) |B_{2i}|}{2i},
\end{equation}
where~$B_{2i}$ denotes the $2i$th Bernoulli number.
See~\cite{oeis} A000367 and A002445.

\begin{proposition}\cite{Milton_book1974, Alzer2000}\label{prop:Milton}
For any positive integer~$i$,
$$
C_i :=\frac{2 (2i)!}{(2\pi)^{2i}}\cdot \frac{1}{1 - 2^{- 2i}} < |B_{2i}| < D_i := \frac{2 (2i)!}{(2\pi)^{2i}} \cdot \frac{1}{1 - 2^{1 - 2i}}
$$
\end{proposition}
Combining~\eqref{eq:AB} with Proposition~\ref{prop:Milton}, we obtain the inequality
\begin{equation}\label{eq:tan}
   \bar{C}_i := \frac{2^{2i} (2^{2i} - 1)}{2i} \cdot C_i
< A_{2i-1} <
 \bar{D}_i := \frac{2^{2i} (2^{2i} - 1)}{2i} \cdot D_i.
\end{equation}

\begin{lemma}\label{lemma:star1}
For each~$k \geq 3$,
$$
a_{k-1}(K_{1,2k})  < a_{k}(K_{1,2k}).
$$
\end{lemma}
\begin{proof}
Substituting~$n = 2k + 1$ and~$i = k - 1, k$ into~\eqref{eq:star}, we obtain
$$
a_{k-1}(K_{1,2k}) = \binom{2k}{2k-3} \cdot A_{2k-3},
\qquad
a_{k}(K_{1,2k}) = \binom{2k}{2k-1} \cdot A_{2k-1}.
$$
Applying~\eqref{eq:tan}, we further have
$$
a_{k-1}(K_{1,2k}) < \binom{2k}{2k-3} \cdot \bar{D}_{k-1},
\qquad
\binom{2k}{2k-1} \cdot \bar{C}_{k} < a_{k}(K_{1,2k}).
$$
It is enough to verify the inequality
\begin{equation}\label{eq:lemineq}
  \binom{2k}{2k-3} \cdot \bar{D}_{k-1} < \binom{2k}{2k-1} \cdot \bar{C}_{k},
\end{equation}
for all~$k \geq 3$.
A straightforward computation shows that~\eqref{eq:lemineq} is equivalent to
\begin{equation}\label{eq:le}
  (2\pi)^{2} \cdot k \cdot (2^{2k-2} - 1)  < 48(k-1)(2^{2k-1}-4).
\end{equation}
For~$k = 3, 4$, and~$5$, this inequality can be verified directly.
Moreover, the following two inequalities
$$
\begin{cases}
  (2\pi)^{2} \cdot k < 48(k-1),\\[2pt]
  2^{2k-2} - 1 < 2^{2k-1}-4,
\end{cases}
$$
hold for all~$k \geq 6$.
This completes the proof.
\end{proof}

\begin{lemma}\label{lemma:star2}
    The $a$-sequence of a star graph~$G = K_{1,n-1}$ is penultimately increasing.
    In particular,
    \begin{equation}\label{eq:lemeq}
      a_0(G) < \cdots < a_{\lfloor \frac{n-1}{2}\rfloor}(G).
    \end{equation}
\end{lemma}

\begin{proof}
    The base cases~$n \leq 6$ can be verified directly.
    We proceed by induction on~$n$.
    Assume that, for some integer~$N \geq 7$, the statement holds for~$n \leq N-1$.
    
    Let~$G$ be a star graph on~$N$ vertices with the universal vertex~$v$, and fix any non-universal vertex~$u$.
    Let~$G'$ be a spanning subgraph of~$G$ obtained by removing the edge~$\edge = \{u,v\}$ of~$G$, so that~$G = G' +\edge$.
    Then~$G'$ is isomorphic to the star graph~$K_{1,N-2}$, and by the induction hypothesis,
    $$
    a_0(G') < \cdots < a_{\lfloor \frac{N-2}{2} \rfloor}(G').
    $$
    
    For each~$J \in \EC(G)\setminus \EC(G')$, we have~$v \in J$, so the reconnected complement~$G^\ast_J$ is a complete graph by Theorem~\ref{thm:col}.
    Hence, by Corollary~\ref{cor:Ham}, the $a$-sequence of each~$G^\ast_J$ is penultimately increasing:
    $$
    a_0(G^\ast_J) \leq \cdots \leq a_{\lfloor \frac{N}{2} \rfloor - j}(G^\ast_J).
    $$
    
    Applying Theorem~\ref{thm:decom}, we conclude that
    $$
    a_0(G) < \cdots < a_{\lfloor \frac{N-2}{2} \rfloor}(G).
    $$
    Finally, Lemma~\ref{lemma:star1} ensures that
    $$
    a_{\lfloor \frac{N-2}{2} \rfloor}(G) < a_{\lfloor \frac{N}{2} \rfloor}(G),
    $$
    so that~\eqref{eq:lemeq} holds for~$G$.
    
    This completes the induction, and hence the proof.
\end{proof}

Applying Theorem~\ref{thm:unimodal}, together with Theorem~\ref{thm:col}~\eqref{thm:col,uni} and Lemma~\ref{lemma:star2}, we deduce the following corollary.
\begin{corollary} \label{cor:universal_vertex}
  If~$G$ has a universal vertex, then the $a$-sequence of~$G$ is penultimately increasing, and therefore is unimodal.
\end{corollary}

It is important to note that not all graphs have a penultimately increasing $a$-vector.
We turn our attention to the path graph~$P_n$ on~$n$ vertices.
As in Example~\ref{example:path}, we have $a_i (P_n) = \frac{n-2i+1}{n-i+1} \binom{n}{i}$.
Hence the consecutive ratio is
$$
    r_i:=\frac{a_{i+1}}{a_i} =\frac{(n-2i-1)(n-i+1)}{(i+1)(n-2i+1)}.
$$
A direct computation shows that $r_i$ is strictly decreasing in $i$ on the range $0\le i<\lfloor s/2\rfloor$ (equivalently, $r_i-r_{i+1}>0$ for all such $i$).
Therefore the sequence $a_i (P_n)$ is log-concave for each $n$, and in particular unimodal.
Moreover, since $r_i\ge 1$ if and only if
$$
    i\le \frac{n-\sqrt{n+2}}{2},
$$
the sequence increases up to an index $i$ close to $(n-\sqrt{n+2})/2$ and then decreases.
Hence, it is not penultimately increasing for sufficiently large $n$.
In fact, this property fails for $n\ge 16$.  

To the best of our knowledge, there is no known example of a graph whose $a$-vector fails to be unimodal.
In fact, it was already conjectured in~\cite{Choi-Yoon2026} that the $a$-sequence is unimodal for every finite simple graph.

\begin{remark} \label{rem:non_log_concave_exam}
    Given this apparent unimodality, a natural question is whether the $a$-sequence is also log-concave.
    However, this does not hold in general.
    For example, consider the star graph~$K_{1,6}$, whose $a$-sequence is given by
    $$
        (a_0,a_1,a_2,a_3,a_4,\ldots) = (1,6,40,96,0,\ldots).
    $$
    This sequence is penultimately increasing and hence unimodal, but it fails to be log-concave, as $a_1^2 - a_0a_2 = -4 < 0$.
\end{remark}



\begin{problem} \label{prob:g-conjecture}
  Characterize the necessary and sufficient conditions for a sequence of positive integers to be realized as the $a$-sequence of a finite simple graph.
\end{problem}

\section{Further geometric remarks}\label{sec:blow-ups}

\subsection{Toric blow-ups and real loci}\label{subsec:blow-ups}

A normal algebraic variety~$X$ is called a \emph{toric variety} if it contains an algebraic torus~$T_X$ as a dense open subset, and the action of~$T_X$ on itself extends to an action on~$X$.
Toric varieties are determined, up to isomorphism, by their associated fans.
We briefly review some basic concepts, following~\cite{Cox_toric_book}.

Let $N$ be a free abelian group of rank~$ n $. Then the real vector space $ N_\R := N \otimes_\Z \R $ has dimension~$ n $. A \emph{cone} $ \sigma \subseteq N_\R $ generated by finitely many elements $ v_1, \dots, v_k \in N $ is defined as
$$
\sigma := \cone (v_1,\ldots,v_k) = \left\{ r_1 v_1 + \cdots + r_k v_k \colon r_1, \dots, r_k \in \R_{\ge 0} \right\}.
$$
If $ \sigma \cap (-\sigma) = \{0\} $, then $ \sigma $ is called a \emph{strongly convex rational polyhedral cone}. When the generating set $ \{v_1, \dots, v_k\} $ is linearly independent over~$ \mathbb{Q} $, the cone is said to be \emph{simplicial}. Moreover, if these generators form part of a basis for~$ N $, then $ \sigma $ is referred to as \emph{smooth}.
Each cone~$ \sigma $ has naturally defined faces. A \emph{fan} $\Sigma$ in $ N $ is a finite collection of strongly convex rational polyhedral cones in $ N_\mathbb{R} $ satisfying the following conditions:
\begin{enumerate}
    \item Every face of a cone in $\Sigma$ is also contained in $\Sigma$.
    \item The intersection of any two cones in $\Sigma$ is a common face of both.
\end{enumerate}
The dimension of a fan $\Sigma$ is the maximum dimension of its cones, which equals $n$ in this setting.
If every cone in $\Sigma$ is smooth (respectively, simplicial), then $\Sigma$ is called a \emph{smooth} (respectively, \emph{simplicial}) fan. Furthermore, if the union of all cones in $\Sigma$ covers the entire space $N_\R$, then $\Sigma$ is said to be \emph{complete}.

By the \emph{fundamental theorem of toric geometry}, there is a one-to-one correspondence between the class of $n$-dimensional toric varieties~$X_\Sigma$ and the class of fans~$\Sigma$ in~$N_\R$.
In particular, $X_\Sigma$ is smooth (respectively, compact) if and only if $\Sigma$ is smooth (respectively, complete).
There is also a one-to-one correspondence between the set of cones in~$\Sigma$ and the set of torus orbits in~$X_\Sigma$, such that the sum of their dimensions equals $\dim X_\Sigma$.
This is known as the \emph{orbit-cone correspondence}.

Let~$X_1$ and~$X_2$ be toric varieties associated with~$\Sigma_1$ and~$\Sigma_2$, respectively.
A morphism~$\varphi \colon X_1 \to X_2$ is \emph{toric} if $\varphi(T_{X_1}) \subseteq T_{X_2}$ and~$\varphi \vert_{T_{X_1}}$ is a group homomorphism.
We assume that~$\Sigma$ is smooth.
Let~$\tau \in \Sigma$ be a cone of dimension~$k$ generated by the primitive vectors~$v_1,\ldots,v_k$, and define
$$
u_\tau := v_1 + \cdots + v_k.
$$
Consider a cone~$\sigma$ such that~$\tau \in \sigma$, 
where $\sigma$ is generated by the primitive vectors~$u_1,\ldots,u_\ell$.
Define
$$
\Sigma_\sigma(\tau)
  = \{\cone(u_\tau,\{u_j \mid j \in A\}) \colon A \subsetneq \{1,\ldots,\ell\}\}.
$$
We define the fan ~$\Sigma^\ast(\tau)$, called the \emph{star subdivision} relative to~$\tau$, as follows:
$$
\Sigma^\ast(\tau) = \{\sigma \in \Sigma \colon \tau \not\subseteq \sigma \} \cup \bigcup_{\tau \subseteq \sigma}\Sigma_\sigma(\tau).
$$
One can check that~$\Sigma^\ast(\tau)$ is smooth.
The star subdivision relative to~$\tau$ induces a toric morphism~$X_{\Sigma^\ast(\tau)} \to X_{\Sigma}$, which is the blow-up of~$X_\Sigma$ along the orbit closure corresponding to~$\tau$.
The toric variety~$X_{\Sigma^\ast(\tau)}$ is called the \emph{toric blow-up} of~$X_\Sigma$ along~$\tau$.

Let~$X$ be a smooth and compact toric variety, and let~$\widetilde{X}$ be a toric blow-up of~$X$.
It is known~\cite{Voisin-Book2002} that the $i$th Betti numbers satisfy
$$
\beta_i(X;\Q) \leq \beta_i(\widetilde{X};\Q)
$$
for all~$i \geq 0$.

A toric variety~$X$ admits a canonical involution induced from a complex conjugation.
The fixed point set~$X^\R$ is called a \emph{real toric variety} associated with~$X$.
This naturally leads to the following question.
\begin{question}\label{ques1}
     Characterize all pairs~$(X,\widetilde{X})$ of smooth and compact toric varieties~$X$ and its toric blow-up~$\widetilde{X}$ such that
    \begin{equation}\label{eq:ine}
      \beta_i(X^\R;\Q) \leq \beta_i(\widetilde{X}^\R;\Q) \quad \text{for all } i \geq 0.
    \end{equation}
\end{question}
In general, the inequality~\eqref{eq:ine} does not hold.
See Example~\ref{ex:per_proj}.
\begin{example}\label{ex:per_proj}
    For odd integer~$n$, consider the $n$-dimensional permutohedral variety~$X_{A_n}$.
    We denote by $\CP^n$ the $n$-dimensional complex projective space.
    The fan of~$X_{A_n}$ is obtained from the fan of~$\CP^n$ by applying successive star subdivisions along all cones of the fan of~$\CP^n$, starting from those of dimension~$n-1$ down to dimension~$1$.
    In particular, this gives rise to a sequence of toric blow-ups
    $$
    X_{A_n} \to \cdots \to \CP^n.
    $$
    The real locus of~$\CP^n$ is~$\RP^n$, and the real locus~$X^\R_{A_n}$ of~$X_{A_n}$ is known~\cite{Henderson2012, Choi-Yoon2023} as the real permutohedral variety.
    Since $n$ is an odd number, we have
    $$
    \beta_n(\RP^n;\Q) = 1 \quad\text{while}\quad  \beta_n(X^\R_{A_n};\Q) = 0,
    $$
    which provides a counterexample to~\eqref{eq:ine}
\end{example}

In the rest of this section, we focus on toric varieties~$X_\B$ associated with building set~$\B$.
See Section~\ref{sec:pre} for relevant definitions.
Let~$\B$ be a connected building set on~$[n+1] = \{1,\ldots,n+1\}$.
A subset~$N \subseteq \B \setminus \{[n+1]\}$ is called a \emph{nested set} if it satisfies the following conditions:
\begin{enumerate}
  \item For each $I,J \in N$ one has either $I \subseteq J$, $J \subseteq I$, or $I \cap J = \emptyset$.
  \item For any collection of $k \geq 2$ disjoint subsets $I_1,\ldots,I_k \in N$, the union $I_1 \cup \cdots \cup I_k \notin \B$.
\end{enumerate}
Let~$e_1, \ldots, e_n$ be the standard basis vectors in~$\R^n$, and~$e_{n+1} := - (e_1 + \cdots + e_n)$.
For any proper subset~$I \subsetneq [n+1]$, define the vector
$$
v_I = \sum_{i \in I}e_i.
$$
Let~$\Sigma_\B$ be the fan whose cone~$\sigma_N$ is generated by the vectors~$\{v_I \colon I \in N\}$ for each nested set~$N \subseteq \B \setminus \{[n+1]\}$.
The fan~$\Sigma_\B$ is called the \emph{nested fan} of dimension~$n$ associated with~$\B$, and it is known~\cite{Zel2006} to be the fan of the toric variety~$X_\B$.

Now, let~$\B$ and~$\B'$ be connected building sets on a finite set~$S$.
If~$\B \subset \B'$, then there is a sequence of toric blow-ups as above:
$$
X_{\B'} \to \cdots \to X_{\B}.
$$
As a refinement of Question~\ref{ques1}, we state the following.
\begin{question}\label{ques2}
    Characterize pairs~$(\B,\B')$ of connected building sets~$\B' \subset \B$ on a finite set~$S$ such that
    \begin{equation}\label{eq:in2}
      \beta_i(X^\R_{\B'};\Q) \leq \beta_i(X^\R_{\B};\Q) \quad \text{for all } i \geq 0.
    \end{equation}
\end{question}
Note that the permutohedral variety~$X_{A_n}$ is the toric variety associated with 
the maximal connected building set~$2^{[n+1]} \setminus \{\emptyset\}$ on~$[n+1]$, 
while the projective space~$\CP^n$ corresponds to the minimal connected building set
\[
\{\{1\},\{2\},\ldots,\{n+1\},[n+1]\}
\]
on~$[n+1]$.
Thus, together with Example~\ref{ex:per_proj}, this shows that the inequality~\eqref{eq:in2} 
does not always hold.

Let~$G$ be a finite simple graph, and let~$\edge$ be an unordered pair of elements of~$V(G)$.
Consider the graphical building sets~$\B_G$ and~$\B_{G+\edge}$.
Then~$\B_{G+\edge}$ is the minimal building set that contains~$\B_G$ as a subset and includes~$\edge$ as an element.
Explicitly,
$$
\B_{G+\edge} = \B_G \cup \{I \subseteq V(G) \colon \edge \subseteq I \text{ and } G\vert_I \text{ is connected.}\}.
$$
From Corollary~\ref{cor:mon}, we obtain the following observation related to Question~\ref{ques2}.
\begin{corollary}
Let~$G$ be a finite simple graph.
If~$H$ is a subgraph of~$G$, then, for each~$i \ge 0$,
$$
\beta_i(X^\R_{H};\Q) \le \beta_i(X^\R_{G};\Q).
$$
\end{corollary}

Some partial results related to Question~\ref{ques2} can be obtained beyond the class of graphical building sets.

\begin{remark}
     For a chordal building set~$\B$ on a finite set~$S$ of positive integers, it is known that the $i$th Betti number of~$X_\B$ equals the number of alternating $\B$-permutations on~$S$ of length~$2i$.
    Consider a pair~$(\B,\B')$ of chordal building sets with~$\B' \subseteq \B$.
    Since every $\B'$-permutation is also a $\B$-permutation, it follows that, for each~$i \geq 0$,
    $$
    \beta_i(X^\R_{\B'};\Q) \leq \beta_i(X^\R_{\B};\Q).
    $$
    For further details, see~\cite{Choi-Yoon2026}.
\end{remark}

\subsection{Non-existence of Lefschetz operators}\label{subsec:Lefschetz}
As observed in Section~\ref{sec:unimodal}, all graphs we have observed exhibit unimodal $a$-sequences.
In the case of projective complex toric varieties, Poincar\'e duality together with the injectivity of the Lefschetz operator
\begin{equation}\label{eq:Lef}
  L \colon H^k(X;\Q) \longrightarrow H^{k+2}(X;\Q)
\end{equation}
plays a crucial role in establishing unimodality.
Motivated by this analogy, we investigate whether real toric varieties associated with graphs admit an analogue of the Lefschetz operator
$$
H^k(X^\R_G;\Q) \longrightarrow H^{k+1}(X^\R_G;\Q).
$$
Just as the operator in~\eqref{eq:Lef} is defined by cupping with a distinguished element of~$H^2(X;\Q)$, namely the symplectic form, one may hope that, if an element in the first cohomology of~$X^\R_G$ playing a similar role could be chosen, it would provide a geometric explanation for the unimodality of the $a$-sequence.

Let~$G$ be a simple graph on a finite set~$S$.
A subset~$N \subseteq \B(G) \setminus \{S\}$ is called a \emph{nested set} if it satisfies the following conditions.
\begin{enumerate}
  \item For each $I,J \in N$ one has either $I \subseteq J$, $J \subseteq I$, or $I \cap J = \emptyset$.
  \item For any collection of $k \geq 2$ pairwise disjoint subsets $I_1,\ldots,I_k \in N$, the union $I_1 \cup \cdots \cup I_k$ does not belong to~$\B(G)$.
\end{enumerate}
The simplicial complex~$K_G$ is defined as the collection of all nested sets of~$\B(G)$ and is called the \emph{nested set complex} of~$\B(G)$.
For an even-cardinality subset $I \subset S$, let $(K_G)_I$ denote the full subcomplex of~$K_G$ induced by all vertices $J$ such that $|J \cap I|$ is odd.
By~\cite{Choi-Park2015,Choi-Park2017_torsion}, we have
\begin{equation}\label{Betti_comp}
  H^k(X_G^\R;\Q) \cong \bigoplus_{\substack{I \subset S \\ \left\vert I \right\vert =2k}} \widetilde{H}^{k-1} ((K_G)_I;\Q) \quad \text{as $\Q$-vector spaces},
\end{equation}
where $\widetilde{H}^\ast((K_G)_I;\Q)$ denotes the reduced (rational) cohomology of the full-subcomplex~$(K_G)_I$.
Moreover, when~$G|_I$ has a connected component of odd order, the complex $(K_G)_I$ is contractible.
Providing a fully explicit combinatorial description of the multiplicative
structure of cohomology ring~$H^\ast(X_G^\R;\Q)$ is quite challenging.
However, it follows from~\cite{Choi-Park2017_multiplicative} that a necessary condition for the product
$$
\smile \colon \widetilde{H}^{k-1} ((K_G)_I;\Q) \otimes \widetilde{H}^{\ell-1} ((K_G)_J;\Q) \to \widetilde{H}^{k+\ell-1} ((K_G)_{I \cup J};\Q)
$$
to be nontrivial is that~$I \cap J \neq \emptyset$.
In the case of a star graph, every subgraph which has no component of odd order should contain the universal vertex, and then, the cup product in cohomology is always trivial.
As this example illustrates, even for~$X^\R_G$ it is not possible to carry out a construction analogous to that used in complex toric varieties to define a Lefschetz operator.

\end{document}